\documentclass{article}
\usepackage{amscd,amsmath, amssymb, amsthm, latexsym, enumerate}
\usepackage{url}

\newtheorem{theorem}{Theorem}[section]

\newtheorem{proposition}[theorem]{Proposition}
\newtheorem{corollary}[theorem]{Corollary}
\theoremstyle{definition}

\theoremstyle{remark}
\newtheorem{remark}[theorem]{Remark}

\numberwithin{equation}{section} \numberwithin{figure}{section}

\title{Graph-induced operators: Hamiltonian cycle enumeration via fermion-zeon convolution}

\author{G. Stacey Staples\footnote{Department of Mathematics and
Statistics, Southern Illinois University Edwardsville,
Edwardsville, IL 62026-1653,USA, email: sstaple@siue.edu}}

\begin{document}
\date{}

\maketitle

\begin{abstract} 
Operators are induced on fermion and zeon algebras by the action of adjacency matrices and combinatorial Laplacians on the vector spaces spanned by the  graph's vertices.  Properties of the algebras automatically give information about the graph's spanning trees and vertex coverings by cycles \& matchings.   Combining the properties of operators induced on fermions and zeons gives a fermion-zeon convolution that recovers the number of Hamiltonian cycles in an arbitrary graph.  The mathematics underlying the graph-theoretic interpretation of these operators is provided by Kirchhoff's theorem and by the seminal works of Goulden and Jackson and Liu, who established formulas for enumeration of Hamiltonian cycles and paths using determinants and permanents of adjacency matrices.  
\\
AMS Subj. Classifications: 05C30, 05E10, 15A66, 47C05, 81R05
\end{abstract}

\section{Introduction \& Notational Preliminaries}

The interface between operator theory and graph theory has been fertile ground for research in recent years.  In the context of quantum probability, the interrelationships have been explored in works by Accardi, Obata, Hashimoto, and others~\cite{ABGO, hashimoto}.   

Fermion algebras are isomorphic to Clifford algebras of appropriate signature.  In the Clifford algebra context, properties of spinors have been applied to the study of maximal cliques~\cite{budinich} and graph enumeration problems~\cite{HarrisStaples2012}.  Although Clifford algebras provide a comprehensive framework for everything appearing here, introducing the more general Clifford algebra formalism is unnecessary for matters at hand.  

Zeon algebras can be thought of as commutative analogues of fermion algebras.  Zeon algebras were defined and applied to graph theory in \cite{ICCA7}, although the name ``zeon'' first appeared in work by Feinsilver~\cite{Feinsilverzeons}.  Weighting the vertices of a graph with zeon generators allows one to construct a {\em nilpotent adjacency matrix}, $\mathfrak{A}$, whose entries are generators of the algebra.  The matrix is very convenient for performing symbolic computations and allows enumeration of cycles by considering traces of matrix powers~\cite{CAMWA, ICCA7}.  

Zeons underlie a significant portion of Liu's work on enumeration of Hamiltonian cycles~\cite{liu}, and their combinatorial properties have been developed in a number of the current author's joint works  (e.g., \cite{JPA,YEUJC,CAMWA}).  Combinatorial properties of zeons are also useful for defining partition-dependent stochastic integrals~\cite{CoSA,JTP}.  In other recent work, Neto and dos Anjos~\cite{NetoSIAM} establish a number of combinatorial identities using Grassmann-Berezin type integration on zeon algebras.   

The main results of the current paper are found in Section \ref{Results Section}.  After a review of essential graph theory and definitions of graph-induced operators, operators are induced by the combinatorial Laplacian and adjacency matrix on the fermion and zeon algebras, respectively.  

In Section \ref{spanning trees section}, the combinatorial Laplacian of a graph $G$ on $n$ vertices is used to induce an operator on the $n$-particle fermion algebra.  In Proposition \ref{spanning trees via fermions}, the normalized trace of the level-$(n-1)$ induced operator is shown to be the number of spanning trees of $G$. 

In Section \ref{cycle-matching section}, the adjacency matrix of $G$ is used to induce an operator on the $n$-particle zeon algebra.  Proposition \ref{induced zeons} reveals the trace of the level-$k$ induced operator to be sums of cycle-matching covers of the vertices of $k$-vertex subgraphs of $G$.  As a corollary, a sort of cycle cover-perfect matching convolution is found as the level-$n$ trace.

Combining the properties of operators induced on fermions and zeons gives a fermion-zeon convolution that recovers the number of Hamiltonian cycles in an arbitrary graph.  As shown in Section \ref{unifying theorem subsection}, this convolution takes the form of an operator trace.  

The combinatorics underlying the graph-theoretic interpretation of these operators is provided by Kirchhoff's theorem and by the seminal works of Goulden and Jackson~\cite{goulden} and  Liu~\cite{liu}, establishing formulas for enumeration of Hamiltonian cycles and paths using determinants and permanents of adjacency matrices.    The interpretation of Liu's formula as fermion-zeon convolution is original with the current author.
 
\section{Fermions \& Zeons}

Denote by $\mathfrak{F}_n$ the associative algebra over $\mathbb{R}$ generated by the collection \hfill\break
 $\{\mathfrak{f}_1, \ldots, \mathfrak{f}_n, {\mathfrak{f}_1}^\dagger, \ldots, {\mathfrak{f}_n}^\dagger\}$ satisfying the canonical anticommutation relations (CAR): \begin{eqnarray}
\{\mathfrak{f}_i, \mathfrak{f}_j\}_+=\{{\mathfrak{f}_i}^\dagger, {\mathfrak{f}_j}^\dagger\}_+&=&0,\label{CAR1}\\
{\mathfrak{f}_j}^2={{\mathfrak{f}_j}^\dagger}^2&=&0,\label{CAR2}\\
\{\mathfrak{f}_i,{\mathfrak{f}_j}^\dagger\}_+&=&\delta_{ij}.
\end{eqnarray}
The algebra $\mathfrak{F}_n$ is called the {\em $n$-particle fermion algebra}.  The generators $\mathfrak{f}_i$ and ${\mathfrak{f}_i}^\dagger$ are referred to as the $i$th {\em annihilation operator} and $i$th {\em creation operator}, respectively.

Zeon algebras can be regarded as commutative subalgebras of fermion algebras.  It is not difficult to see that \eqref{CAR1} implies commutativity of disjoint pairs of fermions; i.e., letting $\zeta_j=\mathfrak{f}_{2j-1}\mathfrak{f}_{2j}\in\mathfrak{F}_{2n}$ for $j=1, \ldots, n$ ensures that the zeons satisfy the zeon canonical commutation relations (CCR): \begin{eqnarray*}
[\zeta_i, \zeta_j]&=&0,\\
{\zeta_j}^2&=&0.
\end{eqnarray*}
Here, $\zeta_j$ is defined as a pair of fermion annihilators, although creators could have been used as easily.

Observe that the collection $\{\zeta_1, \ldots, \zeta_n\}$ with multiplication satisfying $\zeta_i\zeta_j=0\Leftrightarrow i=j$ and satisfying the zeon CCR generates a semigroup of order $2^n-1$.  Appending the unit scalar $\zeta_\emptyset=1$, one obtains a semigroup $\mathcal{Z}_n$ of order $2^n$.  Adopting multi-index notation, one writes $\zeta_I=\prod_{\ell\in I}\zeta_\ell$ for any multi-index $I\subseteq[n]$.  

The {\em $n$-particle zeon algebra} is defined as the semigroup algebra $\mathfrak{Z}_n:=\mathbb{R}\mathcal{Z}_n$.~\footnote{Because the zeon algebra is a subalgebra of a Clifford algebra of appropriate signature, it is often denoted by ${\mathcal{C}\ell_n}^{\rm nil}$.}  Employing multi-index notation, arbitrary elements of $\mathfrak{Z}_n$ are written in the form $u=\displaystyle\sum_{I\subseteq[n]}u_I\,\zeta_I$.  Combinatorial properties of zeons have been developed in a number of works in recent years\cite{Feinsilverzeons, FeinsilverMcSorleyIJC, CoSA,ICCA7}.

\section{Graph-Induced Operators}\label{Results Section}

Assume $G=(V,E)$ is a graph on $n$ vertices.  For $i=1, \ldots, n$ associate vertex $\mathbf{v}_i\in V$ with the $i$th fermion creation/annihilation pair, $\gamma_i=\displaystyle\frac{1}{\sqrt{2}}(\mathfrak{f}_i+{\mathfrak{f}_i}^\dagger)\in\mathfrak{F}_n$.  Note that the fermion CAR imply $\{\gamma_i, \gamma_j\}_+=0$ and ${\gamma_i}^2=1$ for $i=1, \ldots, n$.  Hence, there is no cause for concern in defining multi-index notation by writing the ordered product\begin{equation*}
\gamma_I:=\prod_{j\in I}\gamma_j=\frac{1}{2^{|I|/2}}\displaystyle\prod_{j\in I}(\mathfrak{f}_j+{\mathfrak{f}_j}^\dagger).
\end{equation*}

\begin{remark}
The collection $\{\gamma_i: 1\le i\le n\}$ generates the $2^n$-dimensional Euclidean Clifford algebra, commonly denoted $\mathcal{C}\ell_n$.  This algebra is isomorphic to {\em fermion toy Fock space}.  
\end{remark}

For convenience, $V$ is regarded as both the vertex set of $G$ and as the vector space generated by the vertices of $G$.  Let $\mathcal{L}(V)$ denote the space of linear operators on $V$.  Beginning with an operator $X\in\mathcal{L}(V)$, the corresponding operator $\Psi\in\mathcal{L}(\mathfrak{F}_n)$ induced by $X$ is defined naturally by linear extension of the following action on basis blades: \begin{equation}
\Psi(\gamma_I):=\prod_{j\in I}X(\gamma_j).
\end{equation}
Operators are induced in similar manner on $\mathfrak{Z}_n$.  For convenience, when $\Psi$ is an operator on the algebra $\mathcal{A}$ induced by the operator $X$ acting on the vector space of generators of $\mathcal{A}$, it will be convenient to write $X\twoheadrightarrow\Psi$.  

Because the algebras $\mathfrak{F}_n$ and $\mathfrak{Z}_n$ have natural grade decompositions, an operator $X$ on $V$ similarly induces  operators $\Psi^{(\ell)}$ and $\Xi^{(\ell)}$, respectively,  on the grade-$\ell$ subspaces of the algebras.  More concisely,  
\begin{equation*}
\Psi^{(\ell)}(\gamma_I)=\begin{cases}
\prod_{j\in I}X(\gamma_j) &|I|=\ell,\\
0&\text{\rm otherwise.}
\end{cases}
\end{equation*}

Using Dirac notation, one writes $\Psi^{(\ell)}(\gamma_I)=\langle \gamma_I\vert\Psi^{(\ell)}$.  The use of Dirac notation in matrix representations is made clear by the following:  Given $n\times n$ matrix $A$, the $i$th row and $j$th column of $A$ are given by $\langle i\vert A\vert j\rangle$.  In this context, the matrix $A$ acts on row vectors by {\em right} multiplication.  While this may be less common in linear algebra, it makes sense when dealing with adjacency matrices, combinatorial Laplacians, and transition matrices associated with Markov chains.

Once again, let $G=(V,E)$ be a graph, where $V$ is regarded both as the vertex set of  $G$ and also the vector space spanned by vertices of $G$.  Let $A$ denote an operator on $V$, suppose $A\twoheadrightarrow\Phi\in\mathcal{L}(\mathfrak{F}_n)$, and suppose $A\twoheadrightarrow\Xi\in\mathcal{L}(\mathfrak{Z}_n)$.  Under the vertex-fermion and vertex-zeon associations $\mathbf{v}_i\mapsto \gamma_i$ and $\mathbf{v}_i\mapsto\zeta_i$, respectively, there should be no ambiguity in adopting the conventions $\langle\mathbf{v}_I\vert \Phi\vert \mathbf{v}_J\rangle =\langle \gamma_I\vert\Phi\vert\gamma_J\rangle$ and $\langle\mathbf{v}_I\vert \Xi\vert \mathbf{v}_J\rangle = \langle \zeta_I\vert\Xi\vert\zeta_J\rangle$.

\subsection{Graph preliminaries}

A \textit{graph} $G=\left(V,E\right)$ is a collection of vertices $V$ and a set $E$ of unordered pairs\footnote{When the edges are ordered pairs, the graph is said to be a {\em directed graph}, or {\em digraph}.} of vertices called {\em edges}.    Two vertices $v_i,v_j\in V$ are {\em adjacent} if there exists an edge $e=\{v_i,v_j\}\in E$.  A graph is \textit{finite} if $V$ and $E$ are finite sets. A \textit{loop} in a graph is an edge of the form $\{v,v\}$.  A graph is said to be \textit{simple} if it contains no loops and no unordered pair of vertices appears more than once in $E$.

A {\em $k$-walk} $\{v_0,\ldots,v_k\}$ in a graph $G$ is a sequence of vertices in $G$ with {\em initial vertex} $v_0$ and {\em terminal vertex} $v_k$ such that there exists an edge $\{v_j,v_{j+1}\}\in E$ for each $0\le j\le k-1$.  A $k$-walk contains $k$ edges.  \index{k-walk} A {\em self-avoiding walk} is a walk in which no vertex appears more than once.  A {\em closed k-walk} is a $k$-walk whose initial vertex is also its terminal vertex.  A {\em $k$-cycle} is a self-avoiding closed $k$-walk with the exception $v_0=v_k$.  A {\em Hamiltonian cycle} is an $n$-cycle in a graph on $n$ vertices; i.e., it contains $V.$ 

A {\em tree} is a connected graph that contains no cycles.  A {\em spanning tree} in $G$ is a subgraph that contains all vertices of $G$ and is a tree.

Given a graph $G=(V,E)$, a {\em matching} of $G$ is a subset $E_1\subset E$ of the edges of $G$ having the property that no pair of edges in $E_1$ shares a common vertex.  The largest possible matching on a graph with $n$ vertices consists of $n/2$ edges, and such a matching is called a {\em perfect matching}\index{graph! matching}.

\subsection{Spanning trees via fermion trace}\label{spanning trees section}

When $G=(V,E)$ is a simple graph on $n$ vertices, the {\em combinatorial Laplacian} of $G$ is the $n\times n$ matrix $L=(\ell_{ij})$ defined by
\begin{equation*}
\ell_{ij}=\begin{cases}
{\rm deg}(v_i)&\text{\rm if }i=j,\\
-1&\text{\rm if }\{v_i, v_j\}\in E,\\
0&\text{\rm otherwise.}
\end{cases}
\end{equation*}
Equivalently, if $D$ is the diagonal matrix of vertex degrees and $A$ is the adjacency matrix of $G$, then $L=D-A$.

Letting $\lambda_0\le \lambda_1\le\cdots\le\lambda_{n-1}$ denote the eigenvalues of $L$, known properties of $L$ include (but are not limited to) the following:

\begin{enumerate}
\item The minimum eigenvalue $\lambda_0$ is always zero.~\footnote{Let $\mathbf{v}_0=(1,\ldots, 1)$ and observe that  $L(\mathbf{v}_0)=\mathbf{0}$.}
\item The number of times zero appears as an eigenvalue is equal to the number of connected components of $G$.

\item When $G$ has multiple connected components, $L$ is block-diagonal.
\end{enumerate}

Let ${\rm Tr}(\cdot)$ denote {\em normalized} operator trace.  In particular, if $A$ acts on a finite-dimensional vector space $V$, then \begin{equation}
{\rm Tr}(A):=\frac{1}{\dim(V)}\sum_{i=1}^{\dim(V)} \langle \mathbf{v}_i\vert A\vert \mathbf{v}_i\rangle
\end{equation}
for any orthonormal basis $\{\mathbf{v}_i: 1\le i\le n\}$ of $V$.

The following well known result of Kirchhoff~\cite{kirchhoff} gives meaning to the normalized trace of operators on $\mathfrak{F}_V$ induced by the combinatorial Laplacian.  It is recalled here without proof.

\begin{theorem}[Kirchhoff]\label{kirchhoff's theorem}
For a given connected graph $G$ with $n$ labeled vertices, let $\lambda_1, \ldots, \lambda_{n-1}$ be the non-zero eigenvalues of its Laplacian matrix. Then the number of spanning trees of $G$ is \begin{equation*}
t_G=\frac{1}{n}\lambda_1\lambda_2\cdots\lambda_{n-1}.
\end{equation*}
Equivalently the number of spanning trees is equal to any cofactor of the Laplacian matrix of $G$.
\end{theorem}

Viewing the Laplacian as an operator on the vector space $V$ spanned by vertices of $G$, the combinatorial properties of fermions now allow numbers of spanning trees to be recovered from the trace of the induced operator.   

\begin{proposition}\label{spanning trees via fermions}
Let $L$ denote the combinatorial Laplacian associated with a finite graph, $G=(V,E)$, and let $\Phi$ be the operator on the fermion algebra $\mathfrak{F}_n$ induced by $L$; i.e.,  $L\twoheadrightarrow \Phi$. Then, the normalized trace of the induced map at level $n-1$ satisfies the following:
\begin{equation*}
{\rm Tr}\left(\Phi^{(n-1)}\right)=\sharp\{\text{\rm spanning trees of } G\}.
\end{equation*}
\end{proposition}

\begin{proof}
First, for any $k\times n$ matrix $B=(b_{ij})$, associate the $i$th row of $B$ with the fermion vector $\mathbf{b}_i:=\displaystyle\sum_{j=1}^n \frac{b_{ij}}{\sqrt{2}}\,(\mathfrak{f}_j+{\mathfrak{f}_j}^\dagger)$.   For multi-index $I$ of cardinality $k$, let $B_I$ denote the coefficient submatrix of $B$ whose columns are indexed by elements of $I$.  The fermion CAR guarantee that the coefficient of $\gamma_I$ in the expansion of $\mathfrak{z}=\displaystyle \prod_{\ell=1}^k \mathbf{b}_\ell$ is then $\langle\mathfrak{z}, \gamma_I\rangle = {\rm det}(B_I)$.   

By definition of $L\twoheadrightarrow\Phi$, one immediately finds $\langle \mathbf{v}_I\vert \Phi^{(n-1)}\vert \mathbf{v}_I\rangle = {\rm \det}(L_I)$, which is a cofactor of the combinatorial Laplacian of $G$.  By Kirchhoff's theorem, the number of spanning trees of a connected graph $G$ on $n$ vertices is equal to any cofactor of the Laplacian of $G$.  Summing over the $n$ diagonal elements of the level $n-1$ induced operator then gives the stated result.
\end{proof}

\subsection{Cycle-matching covers via zeon trace}\label{cycle-matching section}

Let $A$ be the adjacency matrix of a graph with vertex set $V$ of cardinality $n$.  Regarding $A$ as a linear operator on ${\rm span}(V)$, an operator $\Xi$ is induced on $\mathfrak{Z}_n$ by multiplication.  Graph-theoretic properties of $\Xi$ are immediately seen.

\begin{proposition}\label{induced zeons}
Let $A$ denote the adjacency matrix of a simple graph with vertex set $V$, viewed as a linear transformation on the vector space generated by $V$.   Let $\Xi^{(k)}$  denote the corresponding operator induced on the grade-$k$ subspace of the zeon algebra $\mathfrak{Z}_V$.  For fixed subset $I\subseteq V$, let $X_I$ denote the number of disjoint cycle covers of the subgraph induced by $I$. Similarly,  let $M_J$ denote the number of perfect matchings on the subgraph induced by $J\subseteq V$ (nonzero only for $J$ of even cardinality).  Then,
\begin{equation*}
{\rm tr}(\Xi^{(k)})=\displaystyle\sum_{{I\subset V}\atop{|I|=k}}\sum_{J\subseteq I}X_{I\setminus J}M_J.
\end{equation*}
\end{proposition}

\begin{proof}
Given any $k\times n$ matrix $B=(b_{ij})$, associate the $i$th row of $B$ with the zeon vector $\mathbf{b}_i:=\displaystyle\sum_{j=1}^n b_{ij}\,\zeta_j$.   For multi-index $I$ of cardinality $k$, let $B_I$ denote the coefficient submatrix of $B$ whose columns are indexed by elements of $I$.  The coefficient of $\zeta_I$ in the expansion of $\mathfrak{z}=\displaystyle \prod_{\ell=1}^k \mathbf{b}_\ell$ is then $\langle\mathfrak{z}, \zeta_I\rangle = {\rm per}(B_I)$.  

For $n\times n$ adjacency matrix $A$, one immediately finds $\langle \mathbf{v}_I\vert \Xi^{(k)}\vert \mathbf{v}_I\rangle = {\rm per}(A_I)$, where $A_I$ is the $k\times k$ submatrix of $A$ whose rows and columns are indexed by elements of $I$.  Writing $A_I=(\Xi_{ij})$ and applying the definition of the matrix permanent, 
\begin{equation}
{\rm per}(A_I)=\sum_{\sigma\in S_k} \prod_{\ell=1}^k \Xi_{\ell \sigma(\ell)}.
\end{equation}
Recall that by an elementary result of group theory, every permutation $\sigma\in S_k$ has a unique (up to order) factorization as a product of disjoint cycles.  Further recall that  $A_I$ is the adjacency matrix of the subgraph induced by vertex set $\mathbf{v}_I$.  Observing that disjoint $2$-cycles (i.e. transpositions) represent edges with no shared vertices (i.e. matchings of subgraphs), it follows that \begin{equation}
{\rm per}(A_I)=\sum_{J\subseteq I}X_{I\setminus J}M_J,
\end{equation}
where $X_{I\setminus J}$ and $M_J$ are defined as in the statement of the theorem.   Summing over all multi-indices of size $k$ then gives the trace of the induced operator.
\end{proof}

The following corollary is immediate.
\begin{corollary}[Cycle-matching convolution]
Let $A$ denote the adjacency matrix of a simple graph with vertex set $V$, and suppose $A\twoheadrightarrow \Xi\in\mathcal{L}(\mathfrak{Z}_n)$ .   Let $X_I$ and $M_{I'}$ be defined as in the statement of Proposition \ref{induced zeons}, where $I'$ is the complement of $I$ in $[n]$.  Then,
\begin{equation*}
\langle \mathbf{v}_{[n]}\vert \Xi\vert \mathbf{v}_{[n]}\rangle = {\rm tr}(\Xi^{(n)})=\displaystyle\sum_{I\subseteq V}X_IM_{I'}.
\end{equation*}
\end{corollary}

\subsection{Hamiltonian cycles via fermion-zeon convolution}\label{unifying theorem subsection}

By combining properties of induced operators on fermions and zeons, one is able to count the Hamiltonian cycles in an arbitrary graph.  The following result of Goulden and Jackson~\cite{goulden} lies at the heart of the main result.  The statement of the theorem has been adapted to the notation developed herein.  

\begin{theorem}[Goulden-Jackson]
Let $H_c$ be the number of directed Hamiltonian circuits in a digraph on $n$ vertices with adjacency matrix $A$.  Then
\begin{equation*}
H_c=\sum_{I} (-1)^{|I|}{\rm det}(A_I){\rm per}(A_{I'})
\end{equation*}
where the sum is over all $I\subseteq [n]\setminus\{c\}$ for any $c\in[n]$, and ${\rm det}(A_\emptyset):=1$.
\end{theorem}

Liu~\cite{liu} generalized Goulden and Jackson's result to obtain a formula somewhat better suited for consideration as fermion-zeon convolution.  In particular, the requirement of summing over subsets omitting one generator can be avoided.  In the following theorem, the graph is assumed to be undirected, and all parameters of Liu's original formulation are assumed to be zero.  Division by 2 is seen as correction for cycles appearing in two orientations.

\begin{theorem}[Liu]\label{liu's theorem}Let $G$ be a finite graph with adjacency matrix $A$, and let $H_c$ denote the number of Hamiltonian cycles in $G$.  Then,
\begin{equation*}
H_c=\frac{1}{2n}\sum_{I\subseteq [n]}(-1)^{n-|I|}|I|{\rm per}(A_I){\rm det}(A_{I'}).
\end{equation*}
\end{theorem}

To make the concept of fermion-zeon convolution rigorous, let $\varphi\in\mathcal{L}(\mathfrak{F}_n)$, let $\xi\in\mathcal{L}(\mathfrak{Z}_n)$, and define combinatorial integrals on $\mathcal{L}(\mathfrak{F}_n)\otimes\mathcal{L}(\mathfrak{Z}_n)$ by  \begin{eqnarray*}
\int (\varphi\ast\xi)(\mathbf{v}_I)\,d\gamma_I\,d\zeta_{I'}&:=&\int(\varphi(\gamma_I)\otimes\xi(\zeta_{I'}))\,d\gamma_I\,d\zeta_{I'}\\
&=&\langle \varphi(\gamma_I)\otimes\xi(\zeta_{I'}), \gamma_I\otimes\zeta_{I'}\rangle,
\end{eqnarray*}
where $I'=[n]\setminus I$ is the complement of $I$ in the $n$-set.  In other words, the value of the integral is the coefficient of $\gamma_I\otimes\zeta_{I'}$ in the canonical expansion of $(\varphi(\gamma_I)\otimes\xi(\zeta_{I'}))\in\mathfrak{F}_n\otimes\mathfrak{Z}_n$.  The {\em fermion-zeon convolution} of $\varphi$ and $\xi$ is then defined by \begin{eqnarray}
\int_{\mathfrak{FZ}} (\varphi\ast\xi)\,d\gamma\,d\zeta:&=& \sum_{I\subseteq[n]}\int(\varphi\ast\xi)(\mathbf{v}_I) \,d\gamma_I\,d\zeta_{I'}\nonumber\\
&=&\sum_{I\subseteq[n]}\langle (\varphi(\gamma_I)\otimes\xi(\zeta_{I'})), \gamma_I\otimes\zeta_{I'}\rangle.
\end{eqnarray}

To see how Liu's result naturally appears as fermion-zeon convolution,  one first defines the {\em star dual} of  $\Xi\in\mathcal{L}(\mathfrak{Z}_n)$ by \begin{equation}
\langle \zeta_I\vert \Xi^\star\vert \zeta_J\rangle:=\langle \zeta_{I'}\vert\Xi\vert\zeta_{J'}\rangle
\end {equation}
for $I,J\subseteq[n]$, where $I'$ and $J'$ denote the set complements of $I$ and $J$, respectively, in the $n$-set.  Second, for operators $X$, $Y$ on spaces of equal finite dimension, denote the componentwise product of operators $X$, $Y$ by $X\odot Y$.  In this way, the $\odot$ product of $\Psi\in\mathcal{L}(\mathcal{C}\ell_Q(V))$ and $\Xi^\star\in\mathcal{L}(\mathfrak{Z}_n)$ satisfies 
\begin{equation}
\langle \mathbf{v}_I\vert\Psi\odot\Xi^\star\vert \mathbf{v}_J\rangle =\langle \mathbf{v}_I\vert \Psi\vert \mathbf{v}_J\rangle\langle \mathbf{v}_I\vert \Xi^\star\vert \mathbf{v}_J\rangle
\end{equation}
for all $I, J\subseteq[n]$.

Letting $\sigma$ denote the diagonal operator on the power set $2^V$ defined by \begin{equation}
\langle \mathbf{v}_I\vert \sigma\vert \mathbf{v}_I\rangle:=(-1)^{|I|}|I'|,
\end{equation}
 the unifying theorem can now be stated.
 
\begin{theorem}[FZ Convolution]\label{hc}
Let $G$ be a simple graph on $n$ vertices having adjacency matrix $A$.  Suppose $A\twoheadrightarrow \Phi\in\mathcal{L}(\mathfrak{F}_n)$ and $A\twoheadrightarrow \Xi\in\mathcal{L}(\mathfrak{Z}_n)$.  Then, the number $H_c$ of Hamiltonian cycles in $G$ is given by the following:
\begin{equation}
H_c=\frac{1}{2n}{\rm tr}(\sigma\Phi\odot\Xi^\star).
\end{equation}
\end{theorem}

\begin{proof}
The proof of Theorem \ref{hc} is now an easy corollary of Theorem \ref{liu's theorem}.  
\begin{eqnarray*}
H_c&=&\frac{1}{2n}\sum_{I\subseteq [n]}(-1)^{n-|I|}|I|{\rm per}(A_I){\rm det}(A_{I'})\\
&=&\frac{1}{2n}\sum_{I\in 2^{[n]}}(-1)^{n-|I|}|I|\langle \mathbf{v}_{I'}\vert \Phi\vert \mathbf{v}_{I'}\rangle\langle \mathbf{v}_{I} \vert \Xi\vert \mathbf{v}_{I}\rangle\\
&=&\frac{1}{2n}\sum_{J\in 2^{[n]}}(-1)^{|J|}|J'|\langle \mathbf{v}_{J}\vert \Phi\vert \mathbf{v}_{J}\rangle\langle \mathbf{v}_{J'} \vert \Xi\vert \mathbf{v}_{J'}\rangle\\
&=&\frac{1}{2n}\sum_{J\in 2^{[n]}}(-1)^{|J|}|J'|\langle \mathbf{v}_{J}\vert \Phi\vert \mathbf{v}_{J}\rangle\langle \mathbf{v}_{J} \vert \Xi^\star\vert \mathbf{v}_{J}\rangle\\
&=&\frac{1}{2n}{\rm tr}(\sigma\Phi\odot\Xi^\star).
\end{eqnarray*}
\end{proof} 
The fermion-zeon convolution interpretation is made clear by observing the following: \begin{eqnarray*}
H_c&=&\frac{1}{2n}\sum_{J\in 2^{[n]}}(-1)^{|J|}|J'|\langle \mathbf{v}_{J}\vert \Phi\vert \mathbf{v}_{J}\rangle\langle \mathbf{v}_{J} \vert \Xi^\star\vert \mathbf{v}_{J}\rangle\\
&=&\frac{1}{2n}\sum_{J\in 2^{[n]}}\langle\sigma(\mathbf{v}_J), \mathbf{v}_J\rangle\int (\Phi(\mathbf{v}_J)\otimes \Xi(\mathbf{v}_{J'}) )(d\gamma_J\otimes d\zeta_{J'})\\
&=&\frac{1}{2n}\sum_{J\in 2^{[n]}}\int (\sigma\Phi(\mathbf{v}_J)\otimes \Xi(\mathbf{v}_{J'}) )(d\gamma_J\otimes d\zeta_{J'})\\
&=&\frac{1}{2n}\int_{\mathfrak{FZ}}(\sigma\Phi\ast \Xi)\, d\gamma\,d\zeta.
\end{eqnarray*}

\section{Conclusion}

The results above lend a deeper physical interpretation of the author's earlier results developed using nilpotent adjacency matrices~\cite{JPA,ICCA7}.   It was shown previously by the author that labeling a simple graph's vertices with zeon generators $\{\zeta_i: i=1, \ldots, n\}$ and constructing the associated  nilpotent adjacency matrix $\mathfrak{A}$ allows one to count Hamiltonian cycles by considering the trace of the $n$th power of $\mathfrak{A}$: \begin{equation*}
{\rm tr}(\mathfrak{A}^n)=2n\,H_c\,\zeta_{[n]}.
\end{equation*}
By Theorem \ref{hc}, the coefficent of $\zeta_{[n]}$ in ${\rm tr}(\mathfrak{A}^n)$ is now given by fermion-zeon convolution as
\begin{equation}
\langle {\rm tr}(\mathfrak{A}^n), \zeta_{[n]}\rangle=\displaystyle{\rm tr}(\sigma\Phi\odot\Xi^\star).
\end{equation}

\subsection*{Acknowledgment}
The author thanks Philip Feinsilver for comments and discussions.

\end{document}